\documentclass[12pt,a4paper,reqno]{amsart}

\usepackage{amssymb}
\usepackage{amsthm}
\usepackage{enumerate}
\usepackage[all]{xy}
\usepackage{fancybox}
\usepackage[colorlinks, linktocpage, citecolor = blue, linkcolor = blue]{hyperref}
\usepackage{graphicx}
\usepackage[rgb]{xcolor}
\usepackage{subcaption}

\newtheorem{theorem}{Theorem}
\newtheorem{proposition}{Proposition}[section]
\newtheorem{lemma}[proposition]{Lemma}

\theoremstyle{definition}
\newtheorem{definition}[proposition]{Definition}
\newtheorem{remark}[proposition]{Remark}

\newtheorem{definitionlemma}[proposition]{Definition and Lemma}

\newcommand{\Proj}{\mathrm{Proj }}
\newcommand{\diag}{\mathrm{diag}}

\newcommand\dashmapsto{\mapstochar\dashrightarrow}
\newcommand{\PP}{\mathbb{P}}
\newcommand{\p}{\mathbb{P}}

\newcommand{\Aut}{\mathrm{Aut}}

\newcommand{\PGL}{\mathrm{PGL}}

\newcommand{\C}{\mathbb{C}}
\newcommand{\R}{\mathbb{R}}
\newcommand{\N}{\mathbb{N}}

\renewcommand{\epsilon}{\varepsilon}
\renewcommand{\phi}{\varphi}

\DeclareMathOperator{\Bir}{Bir}

\newcommand{\Dec}{\mathrm{Dec}}
\newcommand{\Ine}{\mathrm{Ine}}
\newcommand\SB[1][\scalebox{0.8}]{#1}

\newcommand{\Bp}{\mathrm{Bp}}

\def\dashmapsto{\mapstochar\dashrightarrow}

\title[The decomposition group of a line in the plane]{The decomposition group of a line in the plane}

\author{Isac Hed\'en and Susanna Zimmermann}

\address{Isac Hed\'en\\
Research Institute for Mathematical Sciences\\
Kyoto University\\
\mbox{Kyoto} 606-8502 Japan}
\email{Isac.Heden@kurims.kyoto-u.ac.jp}

\address{Susanna Zimmermann\\
Departement Mathematik und Informatik\\
Universit\"at Basel\\
Spiegelgasse 1\\
  4051 Basel, Switzerland}
\email{Susanna.Zimmermann@unibas.ch}

\thanks{The first named author is an International Research Fellow of the Japanese Society for the Promotion of Sciences, and this work was supported by Grant-in-Aid for JSPS Fellows Number 15F15751. Both authors acknowledge support by the Swiss National Science Foundation Grant ''Birational Geometry'' PP00P2\_153026.}
\date{\today}

\begin{document}
\begin{abstract}{We show that the decomposition group of a line $L$ in the plane, i.e. the subgroup of plane birational transformations that send $L$ to itself birationally, is generated by its elements of degree 1 and one element of degree 2, and that it does not decompose as a non-trivial amalgamated product.}\end{abstract}

\maketitle

\section{Introduction}
We denote by $\Bir(\PP^2)$ the group of birational transformations of the projective plane $\p^2=\Proj(k[x,y,z])$, where $k$ is an algebraically closed field. Let $C\subset\p^2$ be a curve, and let $$\Dec(C)=\{\varphi\in\Bir(\p^2),\,\varphi(C)\subset C \textrm{ and } \varphi|_C:C\dasharrow C\textrm{ is birational} \}.$$
This group has been studied for curves of genus $\geq 1$ in \cite{BPV2009}, where it is linked to the classification of finite subgroups of $\Bir(\PP^2)$. It has a natural subgroup $\Ine(C)$, the \emph{inertia group} of $C$, consisting of elements that fix $C$, and Blanc, Pan and Vust give the following result: for any line $L\subset\p^2$, the action of $\Dec(L)$ on $L$ induces a split exact sequence
\[0\longrightarrow \text{Ine}(L)\longrightarrow\Dec(L)\longrightarrow\PGL_2=\Aut(L)\longrightarrow0\]
and $\text{Ine}(L)$ is neither finite nor abelian and also it doesn't leave any pencil of rational curves invariant \cite[Proposition 4.1]{BPV2009}. Further they ask the question whether $\Dec(L)$ is generated by its elements of degree 1 and 2 \cite[Question 4.1.2]{BPV2009}.

We give an affirmative answer to their question in the form of the following result, similar to the Noether-Castelnuovo theorem \cite{Cas} which states that $\Bir(\PP^2)$ is generated by $\sigma\colon[x:y:z]\dashmapsto[yz:xz:xy]$ and $\Aut(\p^2)=\mathrm{PGL}_3$.

\begin{theorem}\label{thm:generators}
For any line $L\subset\PP^2$, the group $\Dec(L)$ is generated by $\Dec(L)\cap\PGL_3$ and any of its quadratic elements having three proper base points in $\PP^2$. 
\end{theorem}
The similarities between $\Dec(L)$ and $\Bir(\p^2)$ go further than this. Cornulier shows in \cite{C13} that $\Bir(\p^2)$ cannot be written as an amalgamated product in any nontrivial way, and we modify his proof to obtain an analogous result for $\Dec(L)$.

\begin{theorem}\label{cor:amalgam}
The decomposition group $\Dec(L)$ of a line $L\subset\PP^2$ does not decompose as a non-trivial amalgam.
\end{theorem}

The article is organised as follows: in Section~\ref{sec:nocontraction} we show that for any element of $\Dec(L)$ we can find a decomposition in $\Bir(\PP^2)$ into quadratic maps such that the successive images of $L$ are curves (Proposition~\ref{Prop:Lnotcontracted}), i.e. the line is not contracted to a point at any time. We then show in Section~\ref{sec:onlylines} that we can modify this decomposition, still in $\Bir(\PP^2)$, into de Jonqui\`eres maps where all of the successive images of $L$ have degree 1, i.e. they are lines. Finally we prove Theorem~\ref{thm:generators}. Our main sources of inspiration for techniques and ideas in Section~\ref{sec:onlylines} have been \cite[\S8.4, \S8.5]{alberich} and \cite{Bla12}. In Section~\ref{sec:noamalgam} we prove Theorem~\ref{cor:amalgam} using ideas that are strongly inspired by \cite{C13}.\\

{\bf Acknowledgement:} The authors would like to thank J\'er\'emy Blanc for helpful discussions, and Yves de Cornulier for kindly answering their questions.

\section{Avoiding to contract $L$}\label{sec:nocontraction}
Given a birational map $\rho\colon\p^2\dasharrow\p^2$, the Noether-Castelnuovo theorem states that there is a decomposition $\rho=\rho_m\rho_{m-1}\dots\rho_1$ of $\rho$ where each $\rho_i$ is a quadratic map with three proper base points. This decomposition is far from unique, and the aim of this section is to show that if $\rho\in\Dec(L)$, we can choose the $\rho_i$ so that none of the successive birational maps $(\rho_i\dots\rho_1\colon\p^2\dasharrow\p^2)_{i=1}^m$ contracts $L$ to a point. This is Proposition~\ref{Prop:Lnotcontracted}.\par
Given a birational map $\phi:X\dasharrow Y$ between smooth projective surfaces, and a curve $C\subset X$ which is contracted by $\phi$, we denote by $\pi_1\colon Z_1\to Y$ the blowup of the point $\phi(C)\in Y$. If $C$ is contracted also by the birational map $\pi_1^{-1}\phi\colon X\dasharrow Z_1$, we denote by $\pi_2\colon Z_2\to Z_1$ the blowup of $(\pi_1^{-1}\phi)(C)\in Z_1$ and consider the birational map $(\pi_1\pi_2)^{-1}\phi\colon X\dasharrow Z_2$. If this map too contracts $C$, we denote by $\pi_3\colon Z_3\to Z_2$ the blowup of the point onto which $C$ is contracted. Repeating this procedure a finite number of times $D\in\N$, we finally arrive at a variety $Z:=Z_D$ and a birational morphism $\pi:=\pi_1\pi_2\cdots\pi_D\colon Z\to Y$ such that $(\pi^{-1}\phi)$ does not contract $C$. Then $(\pi^{-1}\phi)|_C\colon C\dasharrow (\pi^{-1}\phi)(C)$ is a birational map.
\begin{definition}
In the above situation, we denote by $D(C,\phi)\in\N$ the minimal number of blowups which are needed in order to not contract the curve $C$ and we say that $C$ is contracted $D(C,\phi)$ times by $\phi$. In particular, a curve $C$ is sent to a curve by $\phi$ if and only if $D(C,\phi)=0$.
\end{definition}
\begin{remark}
The integer $D(C,\phi)$ can equivalently be defined as the order of vanishing of $K_Z-\pi^*(K_Y)$ along $(\pi^{-1}\phi)(C)$.
\end{remark}

We recall the following well known fact, which will be used a number of times in the sequel.

\begin{lemma}\label{lem:connecting}
Let $\phi_1,\phi_2\in\Bir(\p^2)$ be birational maps of degree $2$ with proper base points $p_1,p_2,p_3$ and $q_1,q_2,q_3$ respectively. If $\phi_1$ and $\phi_2$ have (exactly) two common base points, say $p_1=q_1$ and $p_2=q_2$, then the composition $\tau= \phi_2\phi_1^{-1}$ is quadratic. Furthermore the three base points of $\tau$ are proper points of $\PP^2$ if and only if $q_3$ is not on any of the lines joining two of the $p_i$.
\end{lemma}

\begin{proof}
The lemma is proved by Figure~\ref{fig:07}, where squares and circles in $\p^2_2$ denote the base points of $\phi_1$ and $\phi_2$ respectively. The crosses in $\p^2_1$ denote the base points of $\phi_1^{-1}$ (corresponding to the lines in $\p^2_2$), and the conics in $\p^2_1$ and $\p^2_2$ denote the pullback of a general line $\ell\in\p^2_3$.

\begin{figure}[h]
\centering
\begin{minipage}{.95\columnwidth}

\def\svgwidth{0.95\textwidth}
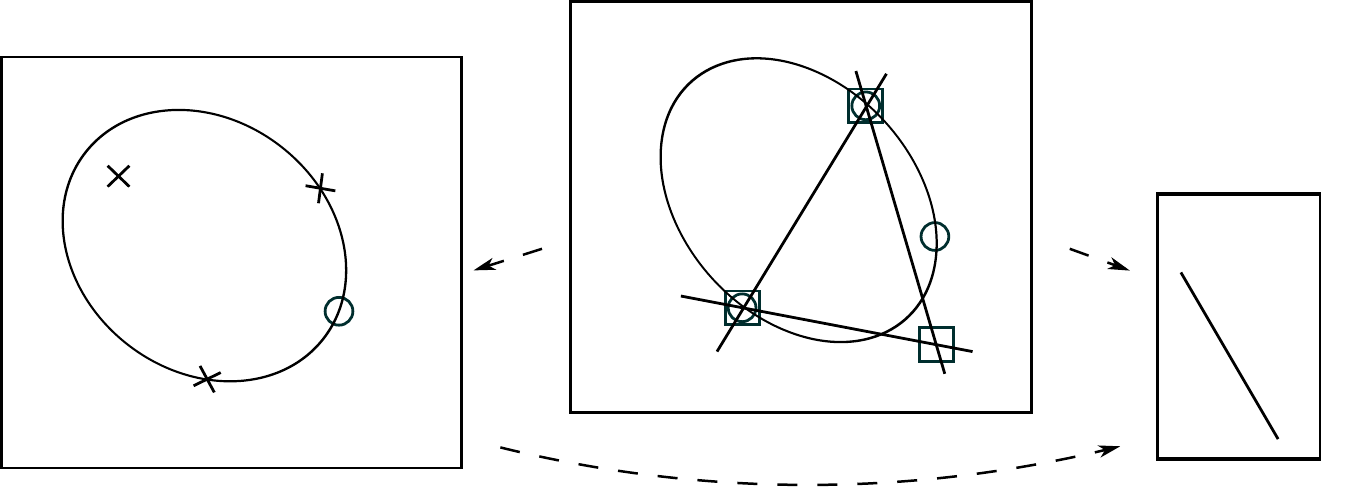
\caption{The composition of $\varphi_1$ and $\varphi_2$ in Lemma~\ref{lem:connecting}}\label{fig:07}
\end{minipage}
\end{figure}
\vspace{.5cm}
If $q_3$ is not on any of the three lines, the base points of $\tau$ are $E_1,E_2,\phi_1(q_3)$. If $q_3$ is on one of the three lines, then the base points of $\tau$ are $E_1,E_2$ and a point infinitely close to the $E_i$ which corresponds to the line that $q_3$ is on.
\end{proof}

The following lemma describes how the number of times that a line is contracted changes when composing with a quadratic transformation of $\p^2$ with three proper base points.
\begin{lemma}\label{lem:4cases} Let $\rho\colon\p^2\dasharrow\p^2$ be a birational map and let $\phi\colon\p^2\dasharrow\p^2$ be a quadratic birational map with base points $q_1,q_2,q_3\in\p^2$. For $1\leq i<j\leq 3$ we denote by $\ell_{ij}\subset\p^2$ the line which joins the base points $q_i$ and $q_j$. If $D(L,\rho)=k\geq 1$, we have
$$D(L,\phi\rho)=\begin{cases}
k+1 &\textrm{ if } \rho(L)\in (\ell_{12}\cup\ell_{13}\cup\ell_{23})\setminus\Bp(\phi),\\
k &\textrm{ if } \rho(L)\notin \ell_{12}\cup\ell_{13}\cup\ell_{23},\\

k &\textrm{ if } \rho(L)=q_i\textrm{ for some } i,\textrm{ and }(\rho\phi)(L)\in\Bp(\phi^{-1}),\\
k-1 &\textrm{ if } \rho(L)=q_i\textrm{ for some } i,\textrm{ and } (\rho\phi)(L)\notin\Bp(\phi^{-1}).
\end{cases}
$$
\end{lemma}

\begin{proof}
We consider the minimal resolutions of $\phi$; in Figures~\ref{fig:02}-\ref{fig:03}, the filled black dots denote the successive images of $L$, i.e. $\rho(L),\,(\pi^{-1}\rho)(L)$ and $(\eta\pi^{-1}\rho)(L)$ respectively.\par
We argue by Figure~\ref{fig:02} and \ref{fig:01} in the case where $\rho(L)$ does not coincide with any of the base points of $\phi$. If $\rho(L)\in\ell_{ij}$ for some $i,j$, then $D(L,\phi\rho)=D(L,\rho)+1$, since $\ell_{ij}$ is contracted by $\phi$. Otherwise, the number of times $L$ is contracted does not change.
\begin{figure}[h]
\begin{tabular}{l@{\hskip .05cm}|@{\hskip .05cm}c}
\centering
\begin{minipage}{0.5\textwidth}
\centering
\def\svgwidth{0.86\textwidth}
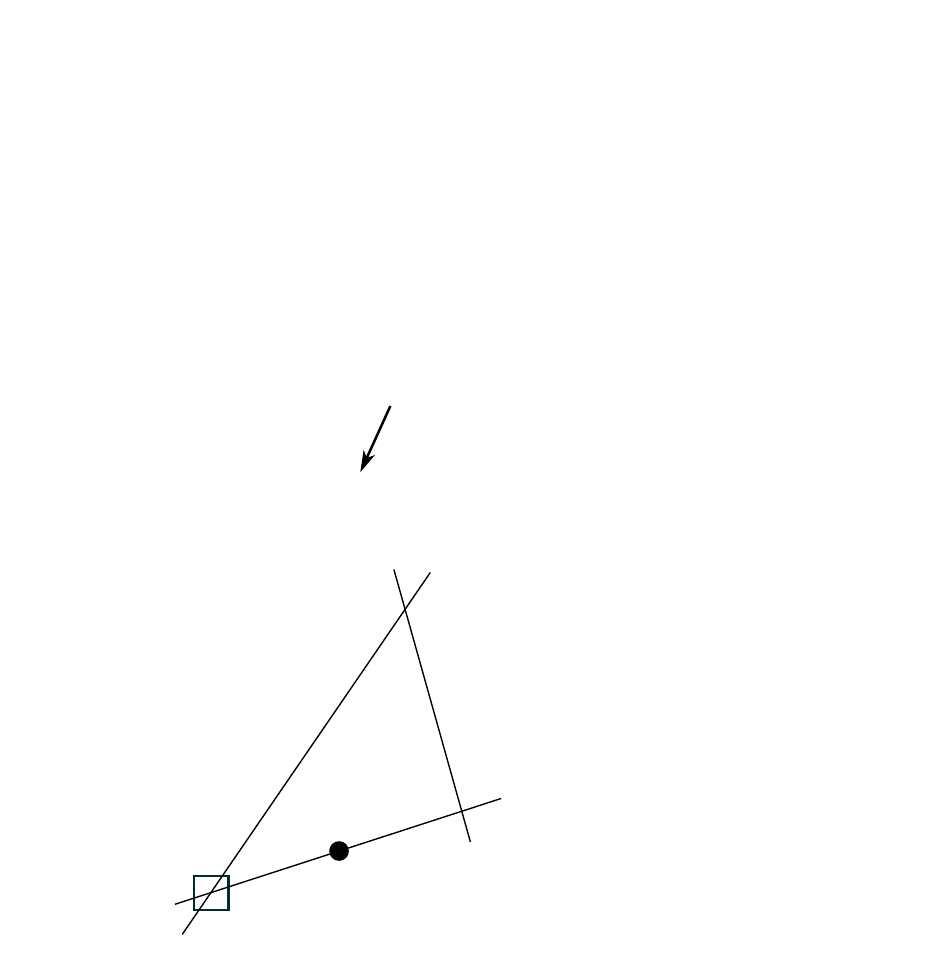
\caption{$D(L,\phi\rho)=k+1$;  $\rho(L)\in (\ell_{12}\cup\ell_{13}\cup\ell_{23})\setminus\Bp(\phi)$.}\label{fig:02}
\end{minipage}
&
\begin{minipage}{0.5\textwidth}
\centering
\def\svgwidth{0.86\columnwidth}
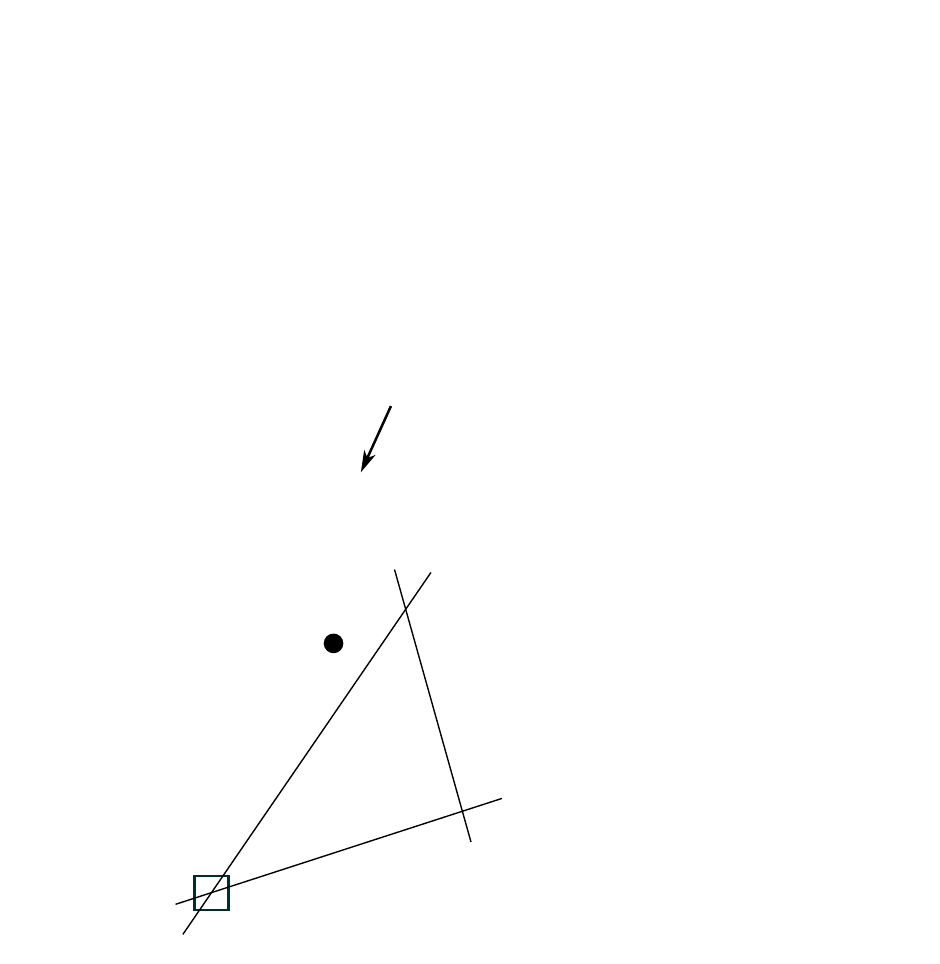
\caption{$D(L,\phi\rho)=k$; $\rho(L)\notin \ell_{12}\cup\ell_{13}\cup\ell_{23}$.}\label{fig:01}
\end{minipage}
\end{tabular}
\end{figure}
Suppose that $\rho(L)=q_i$ for some $i$. If $D(L,\rho)=1$, we have $(\pi^{-1}\rho)(L)=E_i$, and then clearly $D(L,\phi\rho)=0$ since $E_i$ is not contracted by $\eta$. If $D(L,\rho)\geq 2$ we argue by the Figures~\ref{fig:04} and \ref{fig:03}.

\begin{figure}[h]
\begin{tabular}{c@{\hskip .05cm}|@{\hskip .05cm}c}
\centering
\begin{minipage}{0.50\textwidth}
\centering\captionsetup{width=.8\linewidth}
\def\svgwidth{0.86\columnwidth}
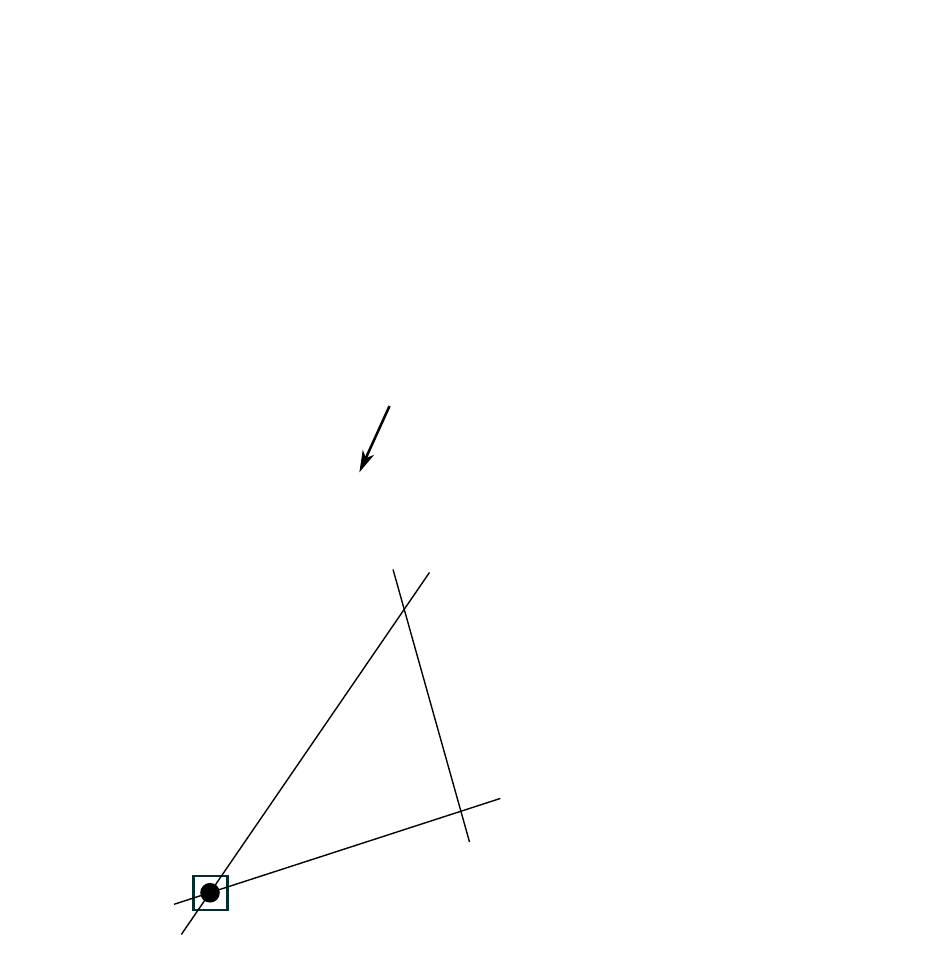
\caption{$D(L,\phi\rho)=k$;\newline $\rho(L)=q_i$ and $(\rho\phi)(L)\in\Bp(\phi^{-1})$.}\label{fig:04}
\end{minipage}
&
\begin{minipage}{0.50\textwidth}
\centering\captionsetup{width=.8\linewidth}
\def\svgwidth{0.86\textwidth}
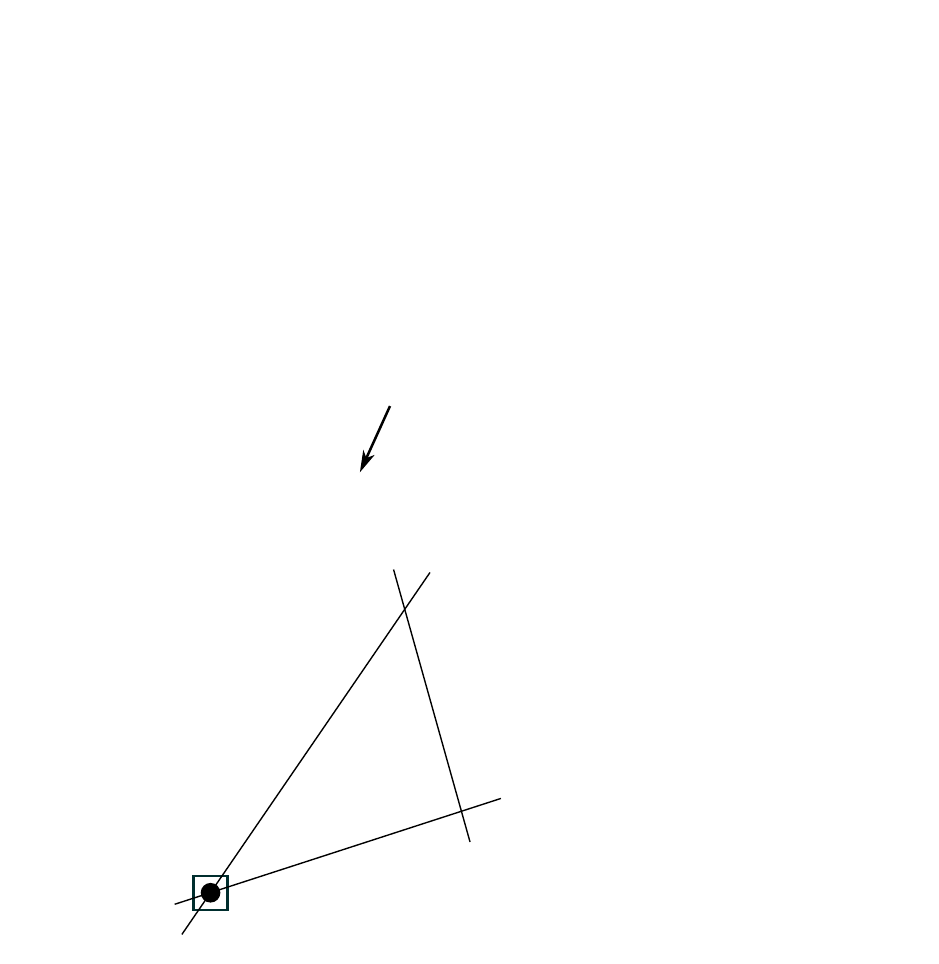
\caption{$D(L,\phi\rho)=k-1$; $\rho(L)=q_i$ and $(\rho\phi)(L)\notin\Bp(\phi^{-1})$.}\label{fig:03}
\end{minipage}
\end{tabular}
\end{figure}
\end{proof}
\begin{remark}\label{rem:Dgeq2} If $D(L,\rho)\geq 2$, then the point $(\pi^{-1}\rho)(L)$ in the first neighbourhood of $\rho(L)$ defines a tangent direction at $\rho(L)\in\p^2$. If we take $\phi$ as in Lemma~\ref{lem:4cases} with $q_i\in\Bp(\phi)$ for some $i$, then this tangent direction coincides with the direction of one of $\ell_{ij},\ell_{ik}$ if and only if $(\rho\phi)(L)\in\Bp(\phi^{-1})$.
\end{remark}
\begin{proposition}\label{Prop:Lnotcontracted}
For any given element $\rho\in\Dec(L)$, there is a decomposition of $\rho$ into quadratic maps $\rho=\rho_m\dots\rho_1$ with three proper base points such that none of the successive compositions $(\rho_i\dots\rho_1)_{i=1}^m$ contract $L$ to a point.
\end{proposition}
\begin{proof}
Let $\rho=\rho_m\dots\rho_1$ be a decomposition of $\rho$ into quadratic maps with only proper base points. We can assume that $d:=\max\{D(L,\rho_j\dots\rho_1)\,\,|\,\,1\leq j\leq m\}>0$, otherwise we are done. Let $n:=\max\{j\,\,|\,\,D(L,\rho_j\dots\rho_1)=d\}$. We denote the base points of $\rho_n^{-1}$ and $\rho_{n+1}$ by $p_1,p_2,p_3$ and $q_1,q_2,q_3$ respectively.

We first look at the case where $D(L,\rho_{n-1}\dots\rho_1)=D(L,\rho_{n+1}\dots\rho_1)=d-1$. Then composition with $\rho_n$ and $\rho_{n+1}$ fall under Cases 1 and 4 of Lemma~\ref{lem:4cases}, so both $\rho_n^{-1}$ and $\rho_{n+1}$ have a base point at $(\rho_n\dots\rho_1)(L)\in\p^2$. We may assume that this point is $p_1=q_1$, as in Figure~\ref{fig:05}. Interchanging the roles of $q_2$ and $q_3$ if necessary, we may assume that $p_1,p_2,q_2$ are not collinear. Let $r\in\p^2$ be a general point, and let $c_1$ and $c_2$ denote quadratic maps with base points $[p_1,p_2,r]$ and $[p_1,q_2,r]$ respectively; then the maps $\tau_1,\tau_2,\tau_3$ (defined by the commutative diagram in Figure~\ref{fig:05}) are quadratic with three proper base points in $\PP^2$. Note that $D(L,\tau_i\dots\tau_1\rho_{n-1}\dots\rho_1)=d-1$ for $i=1,2,3$. Thus we obtained a new decomposition of $\rho$ into quadratic maps with three proper base points $$\rho=\rho_m\dots\rho_{n+2}\tau_3\tau_2\tau_1\rho_{n-1}\dots\rho_1,$$ where the number of instances where $L$ is contracted $d$ times has decreased by 1.

\begin{figure}[h]
\begin{minipage}{0.95\textwidth}
\centering
\def\svgwidth{0.95\columnwidth}
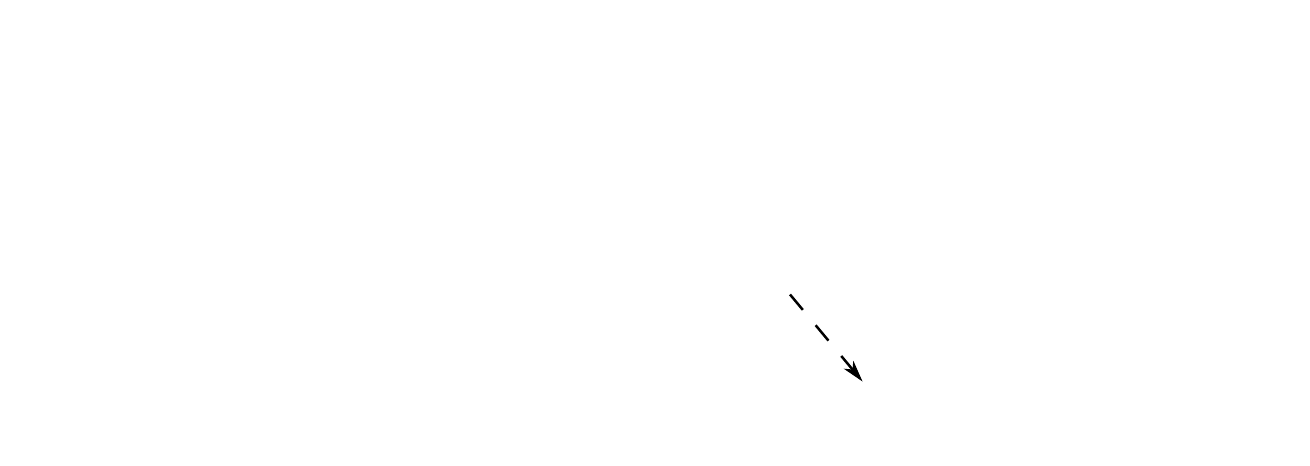
\caption{The decomposition of $\rho_{n+1}\rho_n$ into quadratic maps $\tau_1,\tau_2,\tau_3$}\label{fig:05}
\end{minipage}
\end{figure}
Now assume instead that $D(L,\rho_{n-1}\dots\rho_1)=d$ and $D(L,\rho_{n+1}\dots\rho_1)=d-1$. Then composition with $\rho_{n+1}$ falls under Case 4 of Lemma~\ref{lem:4cases}, so $(\rho_n\dots\rho_1)(L)$ is a base point of $\rho_{n+1}$, which we may assume to be $q_1$. Furthermore composition with $\rho_n$ falls under Cases 2 or 3 of Lemma~\ref{lem:4cases}, so $(\rho_n\dots\rho_1)(L)$ either does not lie on a line joining two base points of $\rho_n^{-1}$, or $D(L,\rho_n\dots\rho_1)\geq 2$ and $(\rho_n\dots\rho_1)(L)$ is a base point of $\rho_n^{-1}$ (which we may assume to be $p_1$, and equal to $q_1$), at the same time as $(\rho_{n-1}\dots\rho_1)(L)$ is a base point of $\rho_n$.\par

We consider the first case. If $D(L,\rho_n\dots\rho_1)\geq 2$ so that $L$ defines a tangent direction at $(\rho_n\dots\rho_1)(L)$, then this tangent direction has to be different from at least two of the three directions at $q_1$ that are defined by the lines through $q_1$ and the $p_i,\,i=1,2,3$. By renumbering the $p_i$, we may assume that $p_2,p_3$ define these two directions (no renumbering is needed if $D(L,\rho_n\dots,\rho_1)=1$). Then with a quadratic map $c_1:=[q_1,p_2,p_3]$ with base points $q_1,p_2,p_3$, we are in Case 4 of Lemma~\ref{lem:4cases} and obtain $D(L,c_1\rho_n\dots\rho_1)=D(L,\rho_n\dots\rho_1)-1$. Let $r,s\in\p^2$ be two general points and define $c_2,c_3,c_4$ with three proper base points respectively as $[q_1,r,p_3],\,[q_1,r,s],\,[q_1,q_2,s]$. Note that the corresponding maps $\tau_1,\dots,\tau_5$, defined in an analogous way as in Figure~\ref{fig:05}, are quadratic with three proper base points. Note also that $D(L,c_i\rho_n\dots\rho_1)=D(L,\rho_n\dots\rho_1)-1$ for $i=2,3,4$. Only for $i=4$ this is not immediately clear, so suppose that this is not the case, i.e. $D(L,c_4\rho_n\dots\rho_1)=D(L,\rho_n\dots\rho_1)$. It follows that $D(L,\rho_n\dots\rho_1)\geq 2$ and that the tangent direction corresponding to $(\rho_n\dots\rho_1)(L)$ is given by the line through $q_1$ and $q_2$, but this is not possible by the assumption that $D(L,\rho_{n+1}\dots\rho_1)=d-1$.\par

In the second case we have $p_1=q_1$ and the tangent direction at $p_1=q_1$ corresponding to $(\rho_n\dots\rho_1)(L)$ is the direction either of the line through $p_1$ and $p_2$ or the line through $p_1$ and $p_3$ (see Figure~\ref{fig:04}). By interchanging the roles of $p_2$ and $p_3$ if necessary, we may assume that it corresponds to the direction of the line through $p_1$ and $p_3$. Interchanging the roles of $q_2$ and $q_3$ if necessary, we may assume that $p_1,q_2,p_3$ are not collinear. Let $r,s\in\p^2$ be general points and define quadratic maps $c_1,c_2,c_3$ with three proper base points respectively by $[p_1,p_2,s],\, [p_1,r,s],\, [p_1,r,q_2]$. Then the corresponding maps $\tau_1,\tau_2,\tau_3,\tau_4$ are quadratic with three proper base points and $D(L,c_i\rho_n\dots\rho_1)=D(L,\rho_n\dots\rho_1)-1$ for $i=1,2,3$. The latter holds for $c_1$ since the direction given by $p_1$ and $p_2$ is different from the tangent direction corresponding to $(\rho_n\dots\rho_1)(L)$, and for $c_3$ it follows from the assumption that the image of $L$ is contracted $d-1$ times by $(\rho_{n+1}\dots\rho_1)$ and that $p_1,q_2,p_3$ are not collinear.\par
Both in the first and second case, we again arrive at a new decomposition into quadratic maps with three proper base points $$\rho=\rho_m\dots\rho_{n+2}\tau_j\dots\tau_1\rho_{n-1}\dots\rho_1\qquad( j\in\{4,5\}),$$ where the number of instances where $L$ is contracted $d$ times has decreased by 1, and we conclude by induction.
\end{proof}

\section{Avoiding to send $L$ to a curve of degree higher than 1.}\label{sec:onlylines}

By Proposition~\ref{Prop:Lnotcontracted}, any element $\rho\in\Dec(L)$ can be decomposed as $$\rho=\rho_m\dots\rho_1$$ where each $\rho_j$ is quadratic with three proper base points, and all of the successive images $((\rho_i\dots\rho_1)(L))_{i=1}^m$ of $L$ are curves. The aim of this section is to show that the $\rho_j$ even can be chosen so that all of these curves have degree 1. That is, we find a decomposition of $\rho$ into quadratic maps such that all the successive images of $L$ are lines. This means in particular that $\Dec(L)$ is generated by its elements of degree 1 and 2.

\begin{definition}
A birational transformation of $\p^2$ is called de Jonqui\`eres if it preserves the pencil of lines passing through $[1:0:0]\in\p^2$. These transformations form a subgroup of $\Bir(\p^2)$ which we denote by $\mathcal{J}$.
\end{definition}

\begin{remark}\label{Rem:dejonq}
In \cite{alberich}, a de Jonqui\`eres map is defined by the slightly less restrictive property that it sends a pencil of lines to a pencil of lines. Given a map with this property, we can always obtain an element in $\mathcal J$ by composing from left and right with elements of $\PGL_3$.
\end{remark}

For a curve $C\subset\PP^2$ and a point $p$ in $\PP^2$ or infinitely near, we denote by $m_C(p)$ the multiplicity of $C$ in $p$. If it is clear from context which curve we are referring to, we will use the notation $m(p)$.

\begin{lemma}\label{Lem:2points}
Let $\phi\in\mathcal{J}$ be of degree $e\geq 2$, and $C\subset\p^2$ a curve of degree $d$. Suppose that 
\[\deg(\phi(C))\leq d.\]
Then there exist two base points $q_1,q_2$ of $\phi$ different from $[1:0:0]$ such that
\[m_C([1:0:0])+m_C(q_1)+m_C(q_2)\geq d.\quad\]
\end{lemma}
This inequality can be made strict in case $\deg(\phi(C))< d$, with a completely analogous proof. 
\begin{proof}
Since $\varphi\in\mathcal{J}$ is of degree $e$, it has exactly $2e-1$ base points $r_0:=[1:0:0],r_1,\dots,r_{2e-2}$ of multiplicity $e-1,1,\dots,1$ respectively. 
Then
\begin{align*}d\geq\deg(\phi(C))=&ed-(e-1)m_C(r_0)-\sum_{i=1}^{e-1}(m_C(r_{2i-1})+m_C(r_{2i}))\\
=&d+\sum_{i=1}^{e-1}(d-m_C(r_0)-m_C(r_{2i-1})-m_C(r_{2i}))
\end{align*}
Hence there exist $i_0$ such that $d\leq m_C(r_0)+m_C(r_{2i_0-1})+m_C(r_{2i_0})$.
\end{proof}
\begin{remark}\label{rmk:choiceofpoints}
Note also that we can choose the points $q_1,q_2$ such that $q_1$ either is a proper point in $\PP^2$ or in the first neighbourhood of $[1:0:0]$, and that $q_2$ either is proper point of $\PP^2$ or is in the first neighbourhood of $[1:0:0]$ or $q_1$. 
\end{remark}

\begin{remark}\label{Rem:decompo} A quadratic map sends a pencil of lines through one of its base points to a pencil of lines, and we conclude from Proposition~\ref{Prop:Lnotcontracted} and Remark~\ref{Rem:dejonq} that there exists maps $\alpha_1,\dots,\alpha_{m+1}\in\PGL_3$ and $\rho_i\in\mathcal J\setminus\PGL_3$ such that $$\rho=\alpha_{m+1}\rho_m\alpha_m\rho_{m-1}\alpha_{m-1}\dots\alpha_2\rho_1\alpha_1$$ and such that all of the successive images of $L$ with respect to this decomposition are curves.
\end{remark}
The following proposition is an analogue of the classical Castelnuovo's Theorem stating that any map in $\Bir(\p^2)$ is a product of de Jonqui\`eres maps. 

\begin{proposition}\label{prop:djdecomp}
Let $\rho\in\Dec(L)$. Then there exists $\rho_i\in\mathcal J\setminus\PGL_3$ and $\alpha_i\in\PGL_3$ such that 
$\rho=\alpha_{m+1}\rho_m\alpha_m\rho_{m-1}\alpha_{m-1}\dots\alpha_2\rho_1\alpha_1$ and all of the successive images of $L$ are lines.
\end{proposition}
\begin{proof}
Start with a decomposition $\rho=\alpha_{m+1}\rho_m\alpha_m\rho_{m-1}\alpha_{m-1}\dots\alpha_2\rho_1\alpha_1$ as in Remark~\ref{Rem:decompo}.

Denote $C_i:=(\rho_i\alpha_i\cdots \rho_1\alpha_1)(L)\subset\p^2$, $d_i:=\deg(C_i)$ and let 
\[D:=\max\{d_i\mid i=1,\dots,m\},\quad n:=\max\{i\mid D=d_i\},\quad k:=\sum_{i=1}^n(\deg \rho_i-1).\]
We use induction on the lexicographically ordered pair $(D,k)$.

We may assume that $D>1$, otherwise our goal is already achieved.
We may also assume that $\alpha_{n+1}\notin\mathcal J$, otherwise the pair $(D,k)$ decreases as we replace the three maps $\rho_{n+1},\alpha_{n+1},\rho_n$ by their composition $\rho_{n+1}\alpha_{n+1}\rho_n\in\mathcal J$. Indeed, either $D$ decreases, or $D$ stays the same while $k$ decreases at least by $\deg\rho_n-1$. 
Using Lemma~\ref{Lem:2points}, we find simple base points $p_1,p_2$ of $\rho_n^{-1}$ and simple base points $\tilde q_1,\tilde q_2$ of $\rho_{n+1}$, all different from $p_0:=[1:0:0]$, such that $$m_{C_n}(p_0)+m_{C_n}(p_1)+m_{C_n}(p_2)\geq D$$ and $$m_{\alpha_{n+1}(C_n)}(p_0)+m_{\alpha_{n+1}(C_n)}(\tilde{q}_1)+m_{\alpha_{n+1}(C_n)}(\tilde{q}_2)>D.$$
We choose $p_1,p_2,\tilde{q}_1,\tilde{q_2}$ as in Remark~\ref{rmk:choiceofpoints}.
By slight abuse of notation, we denote by $q_0=\alpha^{-1}_{n+1}(p_0),\,q_1=\alpha^{-1}_{n+1}(\tilde q_1)$ and $q_2=\alpha^{-1}_{n+1}(\tilde q_2)$ respectively the (proper or infinitely near) points in $\p^2$ that correspond to $p_0,\tilde q_1$, and $\tilde q_2$ under the isomorphism $\alpha_{n+1}^{-1}$. Note that $p_0$ and $q_0$ are two distinct points of $\p^2$ since $\alpha_{n+1}\notin \mathcal J$. We number the points so that $m(p_1)\geq m(p_2)$, $m(\tilde q_1)\geq m(\tilde q_2)$ and so that if $p_i$ (resp. $\tilde q_i$) is infinitely near $p_j$ (resp. $\tilde q_j$), then $j<i$.\par

We study two cases separately depending on the multiplicities of the base points.\par
Case (a): $m(q_0)\geq m(q_1)$ and $m(p_0)\geq m(p_1)$. Then we find two quadratic maps $\tau',\tau\in\mathcal J$ and $\beta\in\PGL_3$ so that $\rho_{n+1}\alpha_{n+1}\rho_n=(\rho_{n+1}\tau^{-1})\beta (\tau\rho_n)$ and so that the pair $(D,k)$ is reduced as we replace the sequence $(\rho_{n+1},\alpha_{n+1},\rho_n)$ by $(\rho_{n+1}\tau^{-1},\beta, \tau\rho_n)$. The procedure goes as follows.\par
If possible we choose a point $r\in\{p_1,q_1\}\setminus\{p_0,q_0\}$. Should this set be empty, i.e. $p_0=q_1$ and $p_1=q_0$, we choose $r=q_2$ instead. The ordering of the points implies that the point $r$ is either a proper point in $\PP^2$ or in the first neighbourhood of $p_0$ or $q_0$. Furthermore, the assumption implies that $m(p_0)+m(q_0)+m(r)>D$, so $r$ is not on the line passing through $p_0$ and $q_0$. In particular, there exists a quadratic map $\tau\in\mathcal J$ with base points $p_0,q_0,r$; then $$\deg(\tau(C_n))=2D-m(p_0)-m(q_0)-m(r)<D.$$ Choose $\beta\in\mathrm{PGL}_3$ so that the quadratic map $\tau':=\beta\tau(\alpha_{n+1})^{-1}$ in the below commutative diagram is de Jonqui\`eres -- this is possible since $\tau$ has $q_0$ as a base point. This decreases the pair $(D,k)$.

\[\xymatrix{
&\p^2\ar@{-->}[dl]_{\rho_n^{-1}}\ar[r]^{\alpha_{n+1}}\ar@{-->}[dd]^{\tau}&\p^2\ar@{-->}[dr]^{\rho_{n+1}}\ar@{-->}[dd]^{\tau'}&\\
\p^2\ar@{-->}[dr]&&&\p^2\\
&\p^2\ar[r]_\beta&\p^2\ar@{-->}[ur]&}\] 

Case (b): $m(p_0)<m(p_1)$. Let $\tau$ be a quadratic de Jonqui\`eres map with base points $p_0,p_1,p_2$. This is possible since our assumption implies that $p_1$ is a proper base point and because $p_0,p_1,p_2$ are base points of $\rho_n^{-1}$ of multiplicity $\deg\rho_n-1,1,1$ respectively and hence not collinear. Choose $\beta_1\in\PGL_3$ which exchanges $p_0$ and $p_1$, let $\gamma=\alpha_{n+1}\beta_1^{-1}$ and choose $\beta_2\in\PGL_3$ so that $\tau':=\beta_2\tau\beta_1^{-1}\in\mathcal J$. The latter is possible since $\beta_1^{-1}(p_0)=p_1$ is a base point of $\tau$, and we have the following diagram.

\[\xymatrix{
&\p^2\ar@{-->}[dl]_{\rho_n^{-1}}\ar[rr]^{\alpha_{n+1}}\ar@{-->}[dd]^{\tau}\ar[dr]^{\beta_1}&&\p^2\ar@{-->}[dr]^{\rho_{n+1}}&\\
\p^2\ar@{-->}[dr]&&\fbox{$\p^2$}\ar[ur]^{\gamma}\ar@{-->}[d]^{\tau'}&&\p^2\\
&\p^2\ar[r]^{\beta_2}&\p^2&&
}\]

Since $\deg(\tau\rho_n)=\deg\rho_n-1$, the pair $(D,k)$ stays unchanged as we replace the sequence $(\alpha_{n+1},\rho_n)$ in the decomposition of $\rho$ by the sequence $(\gamma,(\tau')^{-1}, \beta_2,\tau\rho_n)$. In the new decomposition of $\rho$ the maps $(\tau')^{-1}$ and $\gamma$ play the roles that $\rho_n$ and $\alpha_{n+1}$ respectively played in the previous decomposition. 
In the squared $\PP^2$, we have
\[ m(p_0)=m(\beta_1(p_1))> m(\beta_1(p_0))=m(p_1).\]
Define $q_0':=\gamma^{-1}(p_0)$, $q_1':=\gamma^{-1}(\tilde{q}_1)$, $q_2':=\gamma^{-1}(\tilde{q}_2)$, and note that $q_0'=\beta_1(q_0)$, $q_1'=\beta_1(q_1)$ and $q_2'=\beta_1(q_2)$. In the new decomposition these points play the roles that $q_0,q_1,q_2$ played in the previous decomposition.

If $m(q_0')\geq m(q_1')$, we continue as in case (a) with the points $p_0,p_1,\beta_1(p_2)$ and $q_0',q_1',q_2'$. 

If $m(q_0')<m(q_1')$, we replace the sequence $(\rho_{n+1},\gamma)$ by a new sequence such that, similar to case (a), the roles  of $q_0'$ and $q_1'$ are exchanged, and we will do this without touching $p_0,p_1,\beta(b_2)$.  The replacement will not change $(D,k)$ and we can apply case (a) to the new sequence.

As $m(q_0')<m(q_1')$, the point $q_1'$ is a proper point of $\PP^2$. Analogously to the previous case, there exists $\sigma\in\mathcal{J}$ with base points $\gamma(q_0')=p_0,\gamma(q_1')=\tilde{q}_1,\gamma(q_2')=\tilde{q}_2$, and there exists $\delta_1\in\mathrm{PGL}_3$ which exchanges $p_0$ and $\tilde{q}_1$. Since $\delta_1^{-1}(p_0)=\tilde{q}_1$ is a base point of $\sigma$, there furthermore exists $\delta_2\in\mathrm{PGL}_3$ such that $\sigma':=\delta_2\sigma\delta_1^{-1}\in\mathcal{J}$. Let $\gamma_2:=\delta_1\gamma$.

\[\xymatrix{
&\p^2\ar@{-->}[dl]_{\rho_n^{-1}}\ar[rrrr]^{\alpha_{n+1}}\ar@{-->}[dd]^{\tau}\ar[dr]^{\beta_1}&&&&\p^2\ar@{-->}[dr]^{\rho_{n+1}}\ar[dl]^{\delta_1}\ar@{-->}[dd]^{\sigma}&\\
\p^2\ar@{-->}[dr]&&\fbox{$\p^2$}\ar[urrr]^{\gamma}\ar@{-->}[d]^{\tau'}\ar[rr]_{\gamma_2}&&\p^2\ar@{-->}[d]_{\sigma'}&&\p^2\ar@{-->}[dl]\\
&\p^2\ar[r]^{\beta_2}&\p^2&&\p^2&\p^2 \ar[l]_{\delta_2}&
}\]

Replacing the sequence $(\rho_{n+1},\gamma)$ with $(\rho_{n+1}\sigma^{-1}, \delta_2^{-1}, \sigma', \delta_1\gamma)$ does not change the pair $(D,k)$. The latest position with the highest degree is still the squared $\p^2$ but in the new sequence we have
\[m(\gamma_2^{-1}(p_0))=m(\beta_1(q_1))>m(\beta_1(q_0))=m(\gamma_2^{-1}(\delta_1(\tilde{q}_1)))\]
Since $p_0,p_1,\beta_1(p_2)$ were undisturbed, the inequality $m(p_0)>m(p_1)$ still holds, and we proceed as in case (a).
\par
In this proof, we have used several different quadratic maps $\tau,\tau',\sigma,\sigma'$. Note that none of these can contract $C$ (or an image of $C$), since quadratic maps only can contract curves of degree 1.
\end{proof}

\begin{remark}\label{rmk:djlines}
Suppose that $\rho\in\mathcal{J}$ preserves a line $L$. Then the Noether-equalities imply that $L$ passes either through $[1:0:0]$ and no other base points of $\rho$, or that it passes through exactly $\deg\rho-1$ simple base points of $\rho$ and not through $[1:0:0]$. 
\end{remark}

\begin{lemma}\label{lem:line1}
Let $\rho\in\mathcal{J}$ be of degree $\geq2$ and let $L$ be a line passing through exactly $\deg\rho-1$ simple base points of $\rho$ and not through $[1:0:0]$. Then there exist $\rho_1,\dots,\rho_i \in\mathcal{J}$ of degree $2$ such that $\rho=\rho_m\cdots\rho_1$ and the successive images of $L$ are lines.
\end{lemma}

\begin{proof}
Note that the curve $\rho(L)$ is a line not passing through $\rho(L)$. Call $p_0:=[1:0:0],p_1,\dots,p_{2d-2}$ the base points of $\rho$. Without loss of generality, we can assume that $p_1,\dots,p_{d-1}$ are the simple base points of $\rho$ that are contained in $L$ and that $p_1$ is a proper base point in $\PP^2$. We do induction on the degree of $\rho$. \par
If there is no simple proper base point $p_i$, $i\geq d$, of $\rho$ in $\p^2$ that is not on $L$, choose a general point $r\in\PP^2$. There exists a quadratic transformation $\tau\in\mathcal{J}$ with base points $p_0,p_1,r$. The transformation $\rho\tau^{-1}\in\mathcal{J}$ is of degree $\deg\rho$ and sends the line $\tau(L)$ (which does not contain $[1:0:0]$) onto the line $\rho(L)$. The point $\rho(r)\in\p^2$ is a base point of $(\rho\tau^{-1})^{-1}$ not on the line $\rho(L)$. \par
So, we can assume that there exists a proper base point of $\rho$ in $\PP^2$ that is not on $L$, lets call it $p_d$. The points $p_0,p_1,p_d$ are not collinear (because of their multiplicities), hence there exists $\tau\in\mathcal{J}$ of degree 2 with base points $p_0,p_1,p_d$. The map $\rho\tau^{-1}\in\mathcal{J}$ is of degree $\deg\rho-1$ and $\tau(L)$ is a line passing through exactly $\deg\rho-2$ simple base points of $\rho\tau^{-1}$ and not through $[1:0:0]$. 
\end{proof}

\begin{lemma}\label{lem:line2}
Let $\rho\in\mathcal{J}$ be of degree $\geq2$ and let $L$ be a line passing through $[1:0:0]$ and no other base points of $\rho$. Then there exist $\rho_1,\dots,\rho_m\in\mathcal{J}$ of degree $2$ such that $\rho=\rho_m\cdots\rho_1$ and the successive images of $L$ are lines.
\end{lemma}

\begin{proof}
Note that the curve $\rho(L)$ is a line passing through $[1:0:0]$. We use induction on the degree of $\rho$. 

Assume that $\rho$ has no simple proper base points, i.e. all simple base points are infinitely near $p_0:=[1:0:0]$. There exists a base point $p_1$ of $\rho$ in the first neighbourhood of $p_0$. Choose a general point $q\in\PP^2$. There exists $\tau\in\mathcal{J}$ quadratic with base points $p_0,p_1,q$. The map $\rho\tau^{-1}\in\mathcal{J}$ is of degree $\deg\rho$ and $\tau(L)$ is a line passing through the base point $p_0$ of $\rho\tau^{-1}$ of multiplicity $\deg\rho-1$ and through no other base points of $\rho\tau^{-1}$. Moreover, the point $\rho(q)$ is a (simple proper) base point of $\tau\rho^{-1}$. Therefore, $\tau\rho^{-1}$ has a simple proper base point in $\PP^2$ and sends the line $\rho(L)$ onto the line $\tau(L)$, both of which pass through $p_0$ and no other base points. 

So, we can assume that $\rho$ has at least one simple proper base point $p_1$. Let $p_2$ be a base point of $\rho$ that is a proper point of $\PP^2$ or in the first neighbourhood of $p_0$ or $p_1$. Because of their multiplicities, the points $p_0,p_1,p_2$ are not collinear. Hence there exists $\tau\in\mathcal{J}$ quadratic with base points $p_0,p_1,p_2$. The map $\rho\tau^{-1}$ is a map of degree $\deg\rho-1$ and $\tau(L)$ is a line passing through $p_0$ and no other base points. 
\end{proof}

\begin{lemma}\label{lem:decompgeneric}
Let $\rho\in\mathcal{J}$ be a map of degree $2$ that sends a line $L$ onto a line. Then there exist quadratic maps $\rho_1,\dots,\rho_n\in\mathcal{J}$ with only proper base points such that
\[\rho=\rho_n\cdots\rho_1,\]
and the successive images of $L$ are lines.
\end{lemma}

\begin{proof} Suppose first that exactly two of the three base points of $\rho$ are proper. We number 
\begin{figure}[h]
\begin{tabular}{cc}
\begin{minipage}{0.5\textwidth}
\vspace{-1.1cm}
\centering\captionsetup{width=.8\linewidth}
\def\svgwidth{\columnwidth}
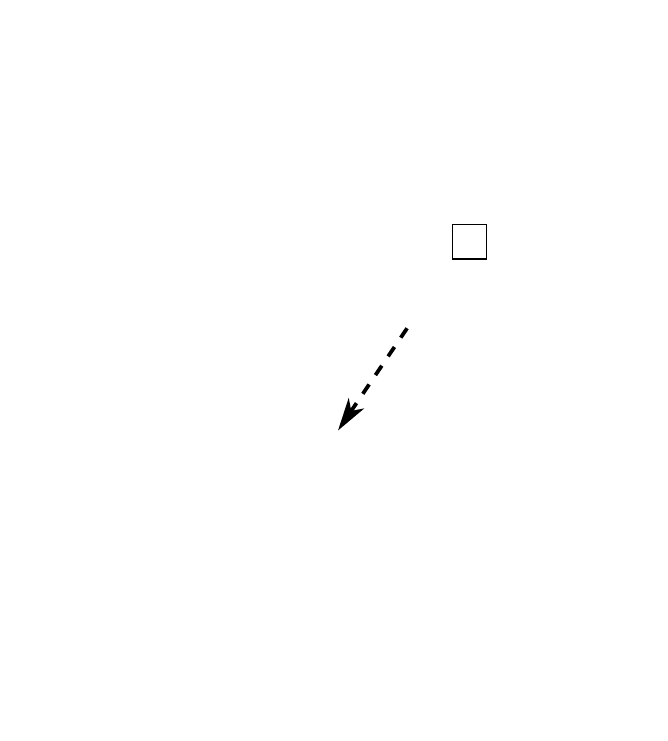
\caption{Numbers in square brackets denote self-intersection.}\label{fig:06}
\end{minipage}
&
\begin{minipage}{0.46\textwidth}
\vspace{-0.92cm}
 the base points so that $p_1,p_2\in\p^2$ and so that $p_3$ is in the first neighbourhood of $p_1$, and denote by $\ell_1\subset\p^2$ the line through $p_1$ which has the tangent direction defined by $p_3$. Choose a general point $r\in\p^2$, and define a quadratic map $\rho_1$ with three base points $p_1,p_2,r\in\p^2$. A minimal resolution of $\rho$ is given by $\pi$ and $\eta$ as in Figure~\ref{fig:06}; it is obtained by blowing up, in order, $p_1,p_2,p_3$, and then contracting in order $\tilde\ell_2:=\eta_*^{-1}(\ell_2)$, $\tilde\ell_1:=\eta_*^{-1}(\ell_1)$ and the exceptional divisor corresponding to $p_1$. By looking at the pull back of a general line in $\p^2$ with respect to $\rho_2:=\rho_1\rho^{-1}$, we see that this map has three proper base points $E_{p_1},\,\rho(r),\,\pi_*(\tilde\ell_1)$. This gives us a decomposition of the desired form: $\rho=\rho_2^{-1}\rho_1.$ Note that since $\rho$ sends the line $L$ onto a line, $L$ has to pass through exactly one of the base points of $\rho$, and this base point  \linebreak
 \end{minipage}
\end{tabular}
\end{figure}

\vspace{-1.05cm}
\noindent has to be proper. Thus $L$ is sent to a line by $\rho_1$. Using the diagram in Figure~\ref{fig:06}, we can see that this line is further sent by $\rho_2^{-1}$ to a line through $E_{p_1}$ if $L$ passes through $p_1$ and a line through $\pi_*(\tilde\ell_1)$ if $L$ passes through $p_2$.\par
 If $[1:0:0]$ is the only proper base point of $\rho$, we reduce to the first case as follows. Denote by $q$ the base point in the first neighbourhood of $[1:0:0]$ and choose a general point $r\in\PP^2$. Let $\rho_1$ be a quadratic map with base points $[1:0:0],q,r$, and let $\rho_2:=\rho_1\rho^{-1}$. If we denote the base points of $\rho^{-1}$ by $q_1,q_2,q_3$ so that $q_1$ is the proper base point and $q_2$ the base point in the first neighbourhood of $q_1$, then the base points of $\rho_2$ are $q_1,q_2,\rho(r)$, i.e. it has exactly two proper base points.\par It is also clear that $\rho_1$ sends $L$ to a line, which is further sent by $\rho_2^{-1}$ to a line through $q_1$. Thus we can apply the first part of this proof to each of $\rho_2^{-1}$ and $\rho_1$ in $\rho=\rho_2^{-1}\rho_1,$ and thus get a decomposition of the desired form.
\vspace{0.15cm}

\end{proof}

\noindent {\bf Theorem \ref{thm:generators}.} {\em
For any line $L$, the group $\Dec(L)$ is generated by $\Dec(L)\cap\PGL_3$ and any of its quadratic elements having three proper base points in $\PP^2$. }
\begin{proof}
By conjugating with an appropriate automorphism of $\PP^2$, we can assume that $L$ is given by $x=y$. Note that the standard quadratic involution $\sigma\colon[x:y:z]\dashmapsto[yz:xz:xy]$ is contained in $\Dec(L)$. 
It follows from Proposition~\ref{prop:djdecomp}, Remark~\ref{rmk:djlines}, and Lemmata~\ref{lem:line1}, \ref{lem:line2} and \ref{lem:decompgeneric} that every element $\rho\in\Dec(L)$ has a composition $\rho=\alpha_{m+1}\rho_m\alpha_m\rho_{m-1}\alpha_{m-1}\cdots\alpha_2\rho_1\alpha_1$, where $\alpha_i\in\mathrm{PGL}_3$ and $\rho_i\in\mathcal{J}$ are quadratic with only proper base points in $\PP^2$ such that the successive images of $L$ are lines. By composing the $\rho_i$ from the left and the right with linear maps, we obtain a decomposition 
\[\rho=\alpha_{m+1}\rho_m\alpha_m\rho_{m-1}\alpha_{m-1}\cdots\alpha_2\rho_1\alpha_1\]
where $\alpha_i\in\mathrm{PGL}_3\cap\Dec(L)$ and $\rho_i\in\Dec(L)$ are of degree $2$ with only proper base points in $\PP^2$.
It therefore suffices to show that for any quadratic element $\rho\in\Dec(L)$ having three proper base points in $\PP^2$ there exist $\alpha,\beta\in\Dec(L)\cap\mathrm{PGL}_3$ such that $\sigma=\beta\rho\alpha$. 

By Remark~\ref{rmk:djlines}, for any quadratic element of $\Dec(L)$ the line $L$ passes through exactly one of its base points in $\PP^2$. 

Let $q_1=[0:0:1]$, $q_2=[0:1:0]$, $q_3=[1:0:0]$. They are the base points of $\sigma$, and $\sigma$ sends the pencil of lines through $q_i$ onto itself. Furthermore, $q_1\in L$ but $q_2,q_3\notin L$. Let $s:=[1:1:1]\in L$. Remark that $\sigma(s)=s$ and that no three of $q_1,q_2,q_3,s$ are collinear.

Let $\rho\in\Dec(L)$ be another quadratic map having three proper base points in $\PP^2$. 
Let $p_1,p_2,p_3$ (resp. $p_1',p_2',p_3'$) be its base points (resp. the ones of $\rho^{-1}$). Say $L$ passes through $p_1$ and $\rho$ sends the pencil of lines through $p_i$ onto the pencil of lines through $p_i'$, $i=1,2,3$. Pick a point $r\in L\setminus\{p_1\}$, not collinear with $p_2,p_3$. Then no three of $p_1,p_2,p_3,r$ (resp. $p_1',p_2',p_3',\rho(r)$) are collinear. In particular, there exist $\alpha,\beta\in\mathrm{PGL}_3$ such that
\[\alpha\colon\begin{cases}q_i\mapsto p_i\\ s\mapsto r\end{cases},\quad\beta\colon\begin{cases}p_i'\mapsto q_i\\ \rho(r)\mapsto s\end{cases}\]
Note that $\alpha,\beta\in\Dec(L)\cap\mathrm{PGL}_3$. Furthermore, the quadratic maps $\sigma,\rho':=\beta\rho\alpha\in\Dec(L)$ and their inverse all have the same base points (namely $q_1,q_2,q_3$) and both $\sigma,\rho'$ send the pencil through $q_i$ onto itself. Since moreover $\rho'(s)=\sigma(s)=s$, we have $\sigma=\rho'$.
\end{proof}

\section{$\Dec(L)$ is not an amalgam}\label{sec:noamalgam}

Just like $\Bir(\PP^2)$, its subgroup $\Dec(L)$ is generated by its linear elements and one quadratic element (Theorem~\ref{thm:generators}). In \cite[Corollary A.2]{C13}, it is shown that $\Bir(\PP^2)$ is not an amalgamated product. In this section we adjust the proof to our situation and prove that the same statement holds for $\Dec(L)$. \par
The notion of being an amalgamated product is closely related to actions on trees, or, in this case, $\R$-trees.

\begin{definitionlemma}
A {\em real tree}, or $\R$-tree, can be defined in the following three equivalent  ways \cite{C01}: 
\begin{enumerate}
\item A geodesic space which is 0-hyperbolic in the sense of Gromov.
\item A uniquely geodesic metric space for which $[a,c]\subset[a,b]\cup[b,c]$ for all $a,b,c$.
\item A geodesic metric space with no subspace homeomorphic to the circle.
\end{enumerate}
We say that a real tree is a {\it complete real tree} if it is complete as a metric space.
\end{definitionlemma}
Every ordinary tree can be seen as a real tree by endowing it with the usual metric but not every real tree is isometric to an simplicial tree (endowed with the usual metric) \cite[\S II.2, Proposition 2.5, Example]{C01}. 

\begin{definition}\label{def:FR}
A group $G$ has the {\it property} $(\mathrm{F}\R)_{\infty}$ if for every isometric action of $G$ on a complete real tree, every element has a fixed point.
\end{definition}
We summarize the discussion in \cite[before Remark A.3]{C13} in the following result.

\begin{lemma}\label{lem:FRamalgam} 
If a group $G$ has property $(\mathrm{F}\R)_{\infty}$, it does not decompose as non-trivial amalgam.
\end{lemma}

We will devote the rest of this section to proving Proposition~\ref{prop:amalgam} and thereby showing that $\Dec(L)$ is not an amalgam.

\begin{proposition}\label{prop:amalgam}
The decomposition group $\Dec(L)$ has property $(\mathrm{F\R})_{\infty}$. 
\end{proposition}

By convention, from now on, $\mathcal{T}$ will denote a complete real tree and all actions on $\mathcal{T}$ are assumed to be isometric. 

\begin{definition}\label{def:rayendstably} Let $\mathcal{T}$ be a complete real tree. 
\begin{enumerate} 
\item A {\em ray} in $\mathcal{T}$  is a geodesic embedding $(x_t)_{t\geq 0}$ of the half-line.
\item An {\it end} in $\mathcal{T}$ is an equivalence class of rays, where we say that two rays $x$ and $y$ are equivalent if there exists $t,t'\in\R$ such that $\{x_s;\,\,s\geq t\}=\{y_s';\,\,s'\geq t'\}$.
\item Let $G$ be a group of isometries of $\mathcal{T}$ and $\omega$ an end in $\mathcal{T}$ represented by a ray $(x_t)_{t\geq 0}$. The group $G$ {\it stably fixes the end} $\omega$ if for every $g\in G$ there exists $t_0:=t_0(g)$ such that $g$ fixes $x_t$ for all $t\geq t_0$. 
\end{enumerate}
\end{definition}

\begin{remark}\label{rmk:FR}\cite[Lemma A.9]{C13}
For a group $G$, property $(\mathrm{F}\R)_{\infty}$ is equivalent to each of the following statements:
\begin{enumerate}
\item\label{A9:1} For every isometric action of $G$ on a complete real tree, every finitely generated subgroup has a fixed point.
\item\label{A9:2} Every isometric action of $G$ on a complete real tree has a fixed point or stably fixes an end.
\end{enumerate}
\end{remark}

\begin{definition}
For a line $L\subset\PP^2$, define $\mathcal{A}_L:=\mathrm{PGL}_3\cap\Dec(L)$. If $L$ is given by the equation $f=0$, we also use the notation $\mathcal{A}_{\{f=0\}}$.
\end{definition}

\begin{lemma}\label{lem:A5}
For any line $L\subset\PP^2$ the group $\mathcal{A}_L$ has property $(\mathrm{F}\R)_{\infty}$. 
\end{lemma}
\begin{proof}
Since for two lines $L$ and $L'$ the groups $\Dec(L)$ and $\Dec(L')$ are conjugate, it is enough to prove the lemma for one line, say the line given by $x=0$. Note that $A=(a_{ij})_{1\leq i,j\leq 3}\in\PGL_3$ is in $\mathcal A_{\{x=0\}}$ if and only if $a_{12}=a_{13}=0$.\par
Let $\mathcal{A}_{\{x=0\}}$ act on $\mathcal{T}$ and let $F\subset \mathcal{A}_{\{x=0\}}$ be a finite subset. The elements of $F$ can be written as a product of elementary matrices contained in $\mathcal{A}_{\{x=0\}}$; let $R$ be the (finitely generated) subring of $k$ generated by all entries of the elementary matrices contained in $\mathcal A_{\{x=0\}}$ that are needed to obtain the elements in $F$. Then $F$ is contained in $\mathrm{EL}_3(R)$, the subgroup of $\mathrm{SL}_3(R)$ generated by elementary matrices. By the Shalom-Vaserstein theorem (see \cite[Theorem 1.1]{EJ10}), $\mathrm{EL}_3(R)$ has Kazhdan's property (T) and in particular (as $\mathrm{EL}_3(R)$ is countable) has a fixed point in $\mathcal{T}$ \cite[Theorem 2]{Wat82}, so $F$ has a fixed point in $\mathcal{T}$. It follows that the subgroup of $\mathcal{A}_{\{x=0\}}$ generated by $F$ has a fixed point \cite[\S I.6.5, Corollary 3]{S77}.
In particular, by Remark~\ref{rmk:FR} (\ref{A9:1}), $\mathcal{A}_{\{x=0\}}$ has property $(\mathrm{F}\R)_{\infty}$. 
\end{proof}

From now on, we fix $L$ to be the line given by $x=y$. It is enough to prove Proposition~\ref{prop:amalgam} for this line since $\Dec(L)$ and $\Dec(L')$ are conjugate groups (by linear elements) for all lines $L$ and $L'$. As before, we denote the standard quadratic involution by $\sigma\in\Bir(\PP^2)$; with our choice of $L$, it is contained in $\Dec(L)$.

Let $\mathcal{D}_L\subset\mathrm{PGL}_3$ be the subgroup of diagonal matrices that send $L$ onto $L$, i.e. $$\mathcal{D}_L:=\left\{\diag(s,s,t)\,s,t\in\C^*\right\}
\subset\mathrm{PGL}_3.$$

\begin{lemma}\label{lem:A6}
We have $\langle \mathcal{D}_L,\mu_1,\mu_2,P\rangle=\mathcal{A}_L$, with the three involutions \[\mu_1:=\begin{bmatrix} -1&0&1\\0&-1&1\\0&0&1\end{bmatrix}\in\mathcal A_L,\, \mu_2:=\begin{bmatrix} -1&0&0\\0&-1&0\\1&0&1\end{bmatrix}\in\mathcal A_L,\,\,\textrm{ and } P:=\begin{bmatrix}0&1&0\\1&0&0\\0&0&1\end{bmatrix} \in\mathcal A_L.\]

\end{lemma}
\begin{proof}
Given any $\lambda\in\C^*$, the matrices 
$$
A_\lambda:=\begin{bmatrix}1&0&0\\0&1&0\\\lambda&0&1\end{bmatrix},\,
B_\lambda:=\begin{bmatrix}1&0&0\\0&1&0\\0&\lambda&1\end{bmatrix},\,\textrm{ and }
C_\lambda:=\begin{bmatrix}1&0&\lambda\\0&1&\lambda\\0&0&1\end{bmatrix}
$$ belong to $\langle \mathcal{D}_L,\mu_1,\mu_2,P\rangle$. Indeed, we have $A_\lambda=\diag(-\lambda^{-1},-\lambda^{-1},1)\cdot\mu_2\cdot\diag(\lambda,\lambda,1),\,B_\lambda=PA_\lambda P$ and $C_\lambda=\diag(1,1,\lambda^{-1})\cdot\mu_1\cdot\diag(-1,-1,\lambda).$\par
Left multiplication by these corresponds to three types of row operations on matrices in $\PGL_3$ and right multiplication corresponds in the same way to three types of column operations. We denote them respectively by $r_1,r_2,r_3,c_1,c_2,c_3$, and we write $d$ for multiplication by an element in $\mathcal D_L$.\par
Let $A=(a_{ij})_{1\leq i,j\leq 3}\in\PGL_3$ be a matrix which is in $\mathcal A_L$, i.e. such that $a_{13}=a_{23}$ and $a_{11}+a_{12}=a_{21}+a_{22}$. We proceed as follows, using only the above mentioned operations. 
\begin{eqnarray*}
A&=&\begin{bmatrix}*&*&*\\ *&*&*\\ *&*&*\end{bmatrix}
\stackrel{d}\longrightarrow\begin{bmatrix}*&*&*\\ *&*&*\\ *&*&1\end{bmatrix}
\stackrel{r_3}\longrightarrow\begin{bmatrix}*&*&0\\ y&z&0\\ *&*&1\end{bmatrix}
\stackrel{c_1\textrm{ and }c_2}\longrightarrow\begin{bmatrix}*&*&0\\ y&z&0\\ -y&-z&1\end{bmatrix}\\
&\stackrel{r_3}\longrightarrow&\begin{bmatrix}*&*&1\\ 0&0&1\\ -y&-z&1\end{bmatrix}
\stackrel{d}\longrightarrow\begin{bmatrix}1&-1&1\\ 0&0&1\\ *&*&1\end{bmatrix}
\stackrel{r_1}\longrightarrow\begin{bmatrix}1&-1&1\\ 0&0&1\\ 0&*&*\end{bmatrix}
\stackrel{r_2}\longrightarrow\begin{bmatrix}1&-1&1\\ 0&0&1\\ 0&*&0\end{bmatrix}\\
&\stackrel{d}\longrightarrow&\begin{bmatrix}1&-1&1\\ 0&0&1\\ 0&1&0\end{bmatrix}
\stackrel{r_3}\longrightarrow\begin{bmatrix}1&0&1\\ 0&1&1\\ 0&1&0\end{bmatrix}
\stackrel{c_3}\longrightarrow\begin{bmatrix}1&0&0\\ 0&1&0\\ 0&1&-1\end{bmatrix}
\stackrel{r_2}\longrightarrow\begin{bmatrix}1&0&0\\ 0&1&0\\ 0&0&-1\end{bmatrix}
\stackrel{d}\longrightarrow\begin{bmatrix}1&0&0\\ 0&1&0\\ 0&0&1\end{bmatrix}
\end{eqnarray*}
In the first step ($d$) we assume that $a_{33}\neq 0$ -- this can always be achieved by performing a row operation of type $r_1$ on $A$ if necessary. In the second step ($r_3$), we use that $a_{13}=a_{23}$. The entries on place $(2,1)$ and $(2,2)$ after the second step are denoted by $y$ and $z$ respectively. In the fifth step ($d$), we use that the entry on place $(1,1)$ is nonzero; this follows from the assumption $a_{11}+a_{12}=a_{21}+a_{22}$ and that $A$ is invertible.
\end{proof}
 
\begin{lemma}\label{lem:A7}
Suppose that $\Dec(L)$ acts on $\mathcal{T}$ so that $\mathcal{A}_L$ has no fixed points. Then $\Dec(L)$ stably fixes an end.
\end{lemma}
\begin{proof}
Since $\mathcal{A}_L$ has property $(\mathrm{F}\mathbb{R})_{\infty}$ and has no fixed points, it stably fixes an end (Remark~\ref{rmk:FR} (\ref{A9:2})). Observe that this fixed end is unique: if $\mathcal{A}_L$ stably fixes two different ends $\omega_1,\omega_2$, then $\mathcal{A}_L$ pointwise fixes the line joining the two ends and has therefore fixed points (this uses that the only isometries on $\R$ are translations and reflections \cite[\S I.2, Lemma 2.1]{C01}).

Let $\omega$, represented by the ray $(x_t)_{t\geq0}$, be the unique end which is stably fixed by $\mathcal{A}_L$ and define $C:=\langle \mathcal{D}_L,P\rangle$. Being a subgroup of $\mathcal{A}_L$, $C$ obviously also stably fixes $\omega$. Note that the end $\sigma\omega$ is stably fixed by $\sigma\mathcal{A}_L\sigma^{-1}$. In particular, since $\sigma C\sigma^{-1}=C$, the end $\sigma\omega$ is also stably fixed by $C$. If $\sigma\omega=\omega$, then $\omega$ is stably fixed by $\sigma$ and by Theorem~\ref{thm:generators}, $\omega$ is stably fixed by $\Dec(L)$. Otherwise, let $l$ be the line joining $\omega$ and $\sigma\omega\neq\omega$. Since $C$ stably fixes $\omega$ and $\sigma\omega$, it stably fixes both ends of $l$. In particular, the line $l$ is pointwise fixed by $C$. Since $\mu_1,\mu_2\in\mathcal{A}_L$, $\mu_1,\mu_2$ stably fix the end $\omega$ and therefore, $x_t$ is fixed by $\mu_1,\mu_2$ for $t\geq t_0$ for some $t_0$, and hence, by Lemma~\ref{lem:A6}, $x_t$ is fixed by all of $\mathcal{A}_L$ for $t\geq t_0$, contradicting the assumption.
\end{proof}

\begin{proof}[Proof of Proposition~$\ref{prop:amalgam}$] 
Recall that $\mu_1,\mu_2\in\mathcal{A}_L$ and note that $\sigma\mu_1$ has order 3 and that $\sigma\mu_2$ has order $6$. It follows that 
\[\sigma=(\mu_1\sigma)\mu_1(\mu_1\sigma)^{-1}\]
By Theorem~\ref{thm:generators}, $\Dec(L)$ is generated by  $\sigma$ and $\mathcal{A}_L$. It follows that $\mathcal{A}_1:=\mathcal{A}_L$ and $\mathcal{A}_2:=\sigma\mathcal{A}_L\sigma$ generate $\Dec(L)$. 

Consider an action of $\Dec(L)$ on $\mathcal{T}$. It induces an action of $\mathcal{A}_L$, which has property $(\mathrm{F}\R)_{\infty}$ by Lemma~\ref{lem:A5} (i.e. $\mathcal{A}_L$ has a fixed point or stably fixes an end by Remark~\ref{rmk:FR} (\ref{A9:2})). If $\mathcal{A}_L$ has no fixed point, Lemma~\ref{lem:A7} implies that $\Dec(L)$ stably fixes an end, and then we are done.\par
Assume that $\mathcal{A}_L$ has a fixed point. We conclude the proof by showing that in this case, even $\Dec(L)$ has a fixed point.\par
For $i=1,2$, let $\mathcal{T}_i$ be the set of fixed points of $\mathcal{A}_i$. The two trees are exchanged by $\sigma$. If $\mathcal{T}_1\cap \mathcal{T}_2\neq\emptyset$, $\Dec(L)$ has a fixed point since $\langle \mathcal{A}_1,\mathcal{A}_2\rangle=\Dec(L)$. Let us consider the case where $\mathcal{T}_1$ and $\mathcal{T}_2$ are disjoint. \par

Let $\mathcal{S}:=[x_1,x_2]$, $x_i\in\mathcal{T}_i$, be the minimal segment joining the two trees and $s>0$ its length. Let $C:=\langle \mathcal{D}_L,P\rangle$. Then $\mathcal{S}$ is pointwise fixed by $C\subset\mathcal{A}_1\cap\mathcal{A}_2$ and reversed by $\sigma$. For $i=1,2$, the image of $\mathcal{S}$ by $\mu_i$ is a segment $\mu_i(\mathcal{S})=[x_1,\mu_ix_2]$. By Lemma~\ref{lem:A6}, $\langle C,\mu_1,\mu_2\rangle=\mathcal{A}_1$, so it follows that for $i=1$ or $i=2$, we have $\mu_i(\mathcal S)\cap \mathcal S=\{x_1\}$. Otherwise, because $\mathcal{T}$ is a tree and $\mathcal A_1$ acts by isometries, both $\mu_1,\mu_2$ fix $\mathcal{S}$ pointwise and so $\mathcal{A}_1$ fixes $\mathcal{S}$ pointwise and in particular it fixes $x_2$ -- this would contradict $\mathcal{T}_1\cap\mathcal{T}_2=\emptyset$. Choose an element $I\in\{1,2\}$ such that $\mu_I(\mathcal{S})\cap\mathcal S=\{x_1\}$. 

Finally we arrive at a contradiction by computing $d(x_1,(\sigma\mu_I)^kx_1)$ in two different ways. On the one hand we see that this distance is $sk$, on the other hand we have $(\sigma\mu_I)^6=1$. More generally, we show that
\[d(\ (\sigma\mu_I)^kx_1,\ (\sigma\mu_I)^lx_1)=| k-l| s\]
for all $k,l$. Since we are on a real tree, it suffices to show this for $k,l$ with $|k-l|\leq2$ (cf. \cite[Lemma A.4]{C13}). By translation, we only have to check it for $l=0,k=1,2$. For $k=1$, we have $d(\sigma\mu_Ix_1,x_1)=d(\sigma x_1,x_1)=d(x_2,x_1)=s$. For $k=2$, the segment $\mu_I(\mathcal{S})=[x_1,\mu_Ix_2]$ intersects $\mathcal{S}$ only at $x_1$. In particular, $d(\mu_Ix_2,x_2)=2s$ and hence 
\[d(\sigma\mu_I\sigma\mu_I x_1,x_1)=d(\sigma\mu_I\sigma x_1,x_1)=d(\mu_I\sigma x_1,\sigma x_1)=d(\mu_I x_2,x_2)=2s.\]
It follows that $\mathcal T_1$ and $\mathcal T_2$ cannot be disjoint, and we are done.
\end{proof}

\vskip\baselineskip

\end{document}